\documentclass[11pt]{amsart}
\usepackage{amsfonts,latexsym,amsthm,amssymb,graphicx}
\usepackage[all]{xy}
\usepackage{amscd}
\DeclareFontFamily{OT1}{rsfs}{}
\DeclareFontShape{OT1}{rsfs}{n}{it}{<-> rsfs10}{}
\DeclareMathAlphabet{\mathscr}{OT1}{rsfs}{n}{it}

\newtheorem{theorem}{Theorem}[section]

\newtheorem{corol}[theorem]{Corollary}
\newtheorem{prop}[theorem]{Proposition}

{\theoremstyle{definition} }
{\theoremstyle{remark} \newtheorem{remark}[theorem]{Remark}

\newcommand{\cO}{{\mathcal {O}}}
\newcommand{\Nbb}{{\mathbb{N}}}
\newcommand{\Pbb}{{\mathbb{P}}}
\newcommand{\Lscr}{\mathscr{L}}
\newcommand{\Escr}{\mathscr{E}}
\newcommand{\Mscr}{\mathscr{M}}
\newcommand{\Fscr}{\mathscr{F}}

\title{On Generalized Sethi-Vafa-Witten Formulas}
\author{James Fullwood}
\address{
Mathematics Department,
Florida State University,
Tallahassee FL 32306, U.S.A.
}
\email{jfullwoo@math.fsu.edu}

\begin{document}
\maketitle

\begin{abstract}
We present a formula for computing proper pushforwards of classes in the Chow ring of a projective bundle under the projection $\pi:\Pbb(\Escr)\rightarrow B$, for $B$ a non-singular compact complex algebraic variety of any dimension. Our formula readily produces generalizations of formulas derived by Sethi,Vafa, and Witten to compute the Euler characteristic of elliptically fibered Calabi-Yau fourfolds used for F-theory compactifications of string vacua. The utility of such a formula is illustrated through applications, such as the ability to compute the Chern numbers of any non-singular complete intersection in such a projective bundle in terms of the Chern class of a line bundle on $B$.
\end{abstract}

\section{Introduction}\label{intro}
Let $B$ be a non-singular compact complex algebraic variety and
$\Lscr$ a line bundle on $B$. We consider a rank $(r+1)$ vector bundle
$\Escr\to B$, which is a direct sum of powers of $\Lscr$, i.e.,
\[
\Escr=\Lscr^{a_1}\oplus \cdots \oplus \Lscr^{a_{r+1}}\quad a_i\in \mathbb{Z}
\]
and its associated projectivization\footnote[1]{Here we take the projective bundle
of \emph{lines} in $\Escr$.} $\Pbb(\Escr)\overset \pi\to B$. Such a projective bundle provides a natural habitat
for \emph{relative} varieties over $B$ with an explicit fibration structure. For  $C\in A^{*}\Pbb(\Escr)$, we
give a formula for computing the canonical pushforward $\pi_*(C)$ in terms of $L=c_1(\Lscr)$. Due to the fact that $\int_{\Pbb(\Escr)}C=\int_{B}\pi_{*}(C)$, such a formula reduces integration on $\Pbb(\Escr)$ to integration on the base $B$, and so is useful for computing topological invariants of varieties in projective bundles (such as Chern numbers, for example). Our motivation comes from F-theory \cite{Vafa}\cite{Denef les houches}, which from a mathematical standpoint is the study of elliptically fibered Calabi-Yau $(n+1)$-folds (which descibe a physical theory of $10-2n$ real space-time dimensions for $n=1,2,3$), which arise naturally as hypersurfaces (or complete intersections) in projective bundles of the form stated above over an $n$-dimensional $B$. Determination of the Euler characteristics of such fibrations are crucial for tadpole cancellation in F-theory, as our formula was derived in an attempt to streamline such calculations. However, knowledge of F-theory will not be assumed, as the results here appear in a more general context.

 By the structure theorem for the Chow \emph{group} of a projective bundle~\cite{fulton}, the map $p:A^{*}B[\zeta]\rightarrow A^{*}\Pbb(\Escr)$ given by
\[
\beta_0+\beta_{1}\zeta+\cdots+\beta_{r}\zeta^{s}\longmapsto \pi^{*}\beta_0+\pi^{*}\beta_{1}\cdot c_1(\cO_{\Pbb(\Escr)}(1))+\cdots+\pi^{*}\beta_{s}\cdot c_1(\cO_{\Pbb(\Escr)}(1))^{s}
\]
is a surjective morphism of rings. Thus any $C\in A^{*}\Pbb(\Escr)$ may be realized as a polynomial in $\zeta$ with coefficients in $A^{*}B$, though not uniquely (e.g. $p(\zeta^{r+dim(B)+1})=0$ in
 $A^{*}\Pbb(\Escr)$). We elaborate more on the structure of $ A^{*}\Pbb(\Escr)$ in \S\ref{structure}. With these facts in mind, we state the main result of this note:

\begin{theorem}\label{mainthm}
With notation and assumptions as above, let $C\in A^{*}\Pbb(\Escr)$, and let $f_{C}= \beta_{0}+\cdots+\beta_{s}\zeta^{s}$ be any such polynomial which maps to $C$ under the natural projection $p:A^{*}B[\zeta]\to A^{*}\Pbb(\Escr)$. Then
\[
\pi_*(C) = \displaystyle\sum_{j=1}^{N}\displaystyle\sum_{i=1}^{m}q_{ij}\frac 1{(j-1)!}\, \frac{d^{j-1}}{d\zeta^{j-1}}(f_{C_{j}}) |_{\zeta=-d_{i}L}
\]
where the $d_i$s are distinct non-zero integers such that
$c(\Escr)=(1+d_{1}L)^{k_{1}}\cdots(1+d_{m}L)^{k_{m}}$, $N=max\{k_{i}\}$, $f_{C_{j}}=\frac{1}{\zeta^{r-j}}\cdot (f_C-(\beta_{0}+\cdots +\beta_{r-1}\zeta^{r-1}))$, and $q_{ij}\in \mathbb{Q}$ are the coefficients of the formal partial fraction decomposition of
$\frac{1}{c(\Escr)}=\frac{1}{(1+d_{1}L)^{k_{1}}\cdots(1+d_{m}L)^{k_{m}}}=\displaystyle\sum_{j=1}^{N}\displaystyle\sum_{i=1}^{m}q_{ij}\frac{1}{(1+d_{i}L)^{j}}$.
\end{theorem}
For many applications, such as computing pushforwards of Chern classes of elliptic fibrations used in F-theory compactifications of string vacua, this deceptively complicated formula involves no more than three summands. The double summation merely reflects the fact that we have to keep track of the partial fraction decomposition of $\frac{1}{c(\Escr)}$. The utility of such a formula lies in the fact that many classes $C\in A^{*}\Pbb(\Escr)$ one encounters in nature are given in terms of rational expressions in $\zeta=c_1(\cO_{\Pbb(\Escr)}(1))$ and pullbacks of classes in $B$, which can easily be expanded as a series in $\zeta$, giving one such a polynomial $f_{C}$. Once such a polynomial is in hand, all that is needed is a partial fraction decomposition of $\frac{1}{c(\Escr)}$ to compute $\pi_*(C)$. As the above discussion might suggest, such a formula lends itself naturally to a simple algorithm which can be easily implemented in computing devices.

We would like to point out that the non-singular hypothesis on $B$ is not absolutely necessary, for if $B$ is singular Theorem~\ref{mainthm} still applies to classes $C$ in the Chow \emph{group} of a projective bundle over $B$, long as the coefficients in the polynomial representation of $C$ described above are pullbacks of Chern classes of vector bundles on $B$, as we will see in \S\ref{structure}.

\emph{Acknowledgements}. The author would like to thank Paolo Aluffi for continued guidance and support, as well as being the impetus from which this work has sprung. He would also like to thank Mirroslav Yotov for proof reading and useful discussions.

\section{Sethi-Vafa-Witten Formulas}\label{svw}
Motivated by tadpole cancellation in F-theory, in~\cite{SVW} Sethi, Vafa, and Witten derived a formula for computing the Euler characteristic of an elliptically fibered Calabi-Yau fourfold in Weierstrass form in terms of the Chern classes of its base. Similar formulas for elliptic fibrations not in Weierstrass form were derived in~\cite{CY4}. In ~\cite{aluffi}, Aluffi and Esole   showed that these formulas are merely avatars of more general Chern class relations, which relate the canonical pushforward of the Chern class of the elliptic fibration with the Chern class of a divisor in the base, and which hold without any Calabi-Yau hypothesis or restrictions on the dimension of the base. A general scheme for producing such formulas was the primary motivation for this note.

\subsection{Elliptic fibrations}\label{EF}

Let $B$ be a non-singular compact complex algebraic variety endowed with a line bundle $\Lscr$. The elliptic fibrations we consider will be constructed by taking general equations of classical elliptic curves over a field $k$, and promoting their coefficients from general elements of $k$ to general sections of appropriate powers of $\Lscr$ (indeed, elliptic curves can be thought of as elliptic fibrations over a point). Such an equation naturally determines an elliptic fibration $\varphi:Y\to B$, in which $Y$ is realized as the zero-scheme of a section of a certain line bundle on $\Pbb(\Escr)$, where $\Escr$ is a direct sum of powers of $\Lscr$. We lift terminology from the physics literature in what follows.

Let $\Escr=\Lscr^{a_1}\oplus \Lscr^{a_2}\oplus \Lscr^{a_3}$. For $(a_1,a_2,a_3)=(0,1,1)$, we define an $E_6$ elliptic fibration over $B$ to be a non-singular hypersurface $Y$ in $\Pbb(\Escr)$, determined by the equation
\[
 E_6: x^3+y^3=dxyz+exz^2+fyz^2+gz^3
\]
where $x$,$y$,and $z$ are sections of $\cO_{\Pbb(\Escr)}(1)\otimes \Lscr$, $\cO_{\Pbb(\Escr)}(1)\otimes \Lscr$, and $\cO_{\Pbb(\Escr)}(1)$ respectively, and $d$, $e$, $f$, and $g$ are chosen to be suitably generic sections of $\Lscr$, $\Lscr^{2}$, $\Lscr^{2}$, and $\Lscr^{3}$ respectively.\footnote[2]{Here and throughout, we often write $\Lscr^a$ when we mean $\pi^{*}\Lscr^a$. The use of this elision should be clear from the context in which it is used.}

Similarly, for $(a_1,a_2,a_3)=(0,1,2)$, an $E_7$ elliptic fibration is a non-singular hypersurface $Y$ in the weighted projectivization $\Pbb_{1,1,2}(\Escr)$, determined by the equation
\[
E_7: y^2=x^4+ex^2z^2+fxz^3+gz^4
\]
where $e$, $f$, and $g$ are chosen to be sections of $\Lscr^{2}$, $\Lscr^{3}$, and $\Lscr^{4}$ respectively, and $x$, $y$, and $z$ are chosen to be sections of appropriate line bundles so that each monomial in the equation for $Y$ is a section of $\cO_{\Pbb(\Escr)}(4)\otimes \Lscr^4$.

The last family of elliptic fibrations we will mention are referred to in the physics literature as $E_8$ elliptic fibrations, determined by the Weierstrass normal equation
\[
E_8: y^{2}z=x^3+fxz^2+gz^3
\]
Here, $(a_1,a_2,a_3)=(0,2,3)$, $f$ and $g$ are sections of $\Lscr^{4}$ and $\Lscr^{6}$ respectively, and again $x$, $y$, and $z$ are chosen to be sections of appropriate line bundles so that each monomial in the equation for $Y$ is a section of $\cO_{\Pbb(\Escr)}(3)\otimes \Lscr^6$. As in the theory of algebraic curves, every non-singular elliptic fibration with a section is birational to a (possibly singular) elliptic fibration in Weierstrass form. Because of this, Weierstrass models of elliptic fibrations are often singled out in the F-theory literature. Though such a model is useful for computing the $j$-invariant as well as the \emph{location} of the singular fibers of such a fibration, singular fibers are not preserved under a birational map. As the singular fibers are crucial not only to the geometry of an elliptic fibration (e.g., only the singular fibers contribute to its Euler characteristic), but the physical interpretation as well (\cite{Vafa}\cite{Denef les houches}\cite{timo}, p.~21), alternate forms of elliptic fibrations not in Weierstrass form such as the $E_6$ and $E_7$ cases listed above (and ones not listed above) should perhaps stand on equal footing from a physical perspective. Be that as it may, we now summarize the results referenced in the beginning of this section, along with an application of Theorem~\ref{mainthm}.
\begin{theorem}\label{CC}(~\cite{aluffi})
Let $\varphi:Y\to B$ be an elliptic fibration of type $E_6$, $E_7$, or $E_8$(as defined above), and $Z$ a smooth hypersurface in $B$ in the same class as $g=0$. Then
\[
\varphi_{*}c(Y)=m\cdot c(Z) ,
\]
where m=4,3,2 respectively for $Y$ of type $E_6$, $E_7$, $E_8$, respectively.

\end{theorem}
Before an illustration of how Theorem~\ref{mainthm} can be used to reproduce such formulas, we state a corollary of Theorem~\ref{CC}, which follows from the fact that taking the degree of a zero dimensional class is invariant under proper pushforwards, i.e., $\int\varphi_{*}c(Y)=\int c(Y)=\chi(Y)$, and the fact that $Y$ is Calabi-Yau if and only if $c_1(B)=c_1(\Lscr)$.\footnote[3]{The reader not familiar with this fact may enjoy proving this using B.5.8 in \cite{fulton}.}
\begin{corol}\label{chi}(\cite{SVW}, \cite{CY4})
Let $Y\to B$ be an elliptic fibration of type $E_6$, $E_7$, or $E_8$, with $Y$ a Calabi-Yau fourfold. Then
\begin{eqnarray*}
                      E_6:\hspace{2mm} \chi(Y)&=&\int12c_1(B)c_2(B)+72c_1(B)^3\\
                      E_7:\hspace{2mm} \chi(Y)&=&\int12c_1(B)c_2(B)+144c_1(B)^3\\
                      E_8:\hspace{2mm} \chi(Y)&=&\int12c_1(B)c_2(B)+360c_1(B)^3\\
\end{eqnarray*}
\end{corol}
\begin{remark}
Though indeed Corollary~\ref{chi} is a consequence of Theorem~\ref{CC}, chronologically Corollary~\ref{chi} appeared first. We would like to think of Theorem~\ref{CC} as a representation of a deeper, more precisely formulated geometric relationship between the fibration and the base than that of Corollary~\ref{chi}.
\end{remark}
We now illustrate the utility of Theorem~\ref{mainthm} with a proof of Theorem~\ref{CC} in the $E_8$ case.

\begin{proof} Let $L=c_1(\Lscr)$, $H=c_1(\cO_{\Pbb(\Escr)}(1))$, and denote the projection $\Pbb(\Escr)\to B$ by $\pi$. By adjunction and B.5.8 in \cite{fulton},
\[
c(Y)=\frac{(1+H)(1+2L+H)(1+3L+H)}{1+3H+6L}(3H+6L)\pi^{*}c(B),
\]
so
\[
\varphi_{*}c(Y)=\pi_{*}\left(\frac{(1+H)(1+2L+H)(1+3L+H)}{1+3H+6L}(3H+6L)\right)c(B)
\]
by the projection formula. Thus computing $\varphi_{*}c(Y)$ amounts to computing
\[
\pi_{*}\left(\frac{(1+H)(1+2L+H)(1+3L+H)}{1+3H+6L}(3H+6L)\right),
\]
which is where Theorem~\ref{mainthm} comes into play. Let $C=\frac{(1+H)(1+2L+H)(1+3L+H)}{1+3H+6L}(3H+6L)\in A^{*}\Pbb(\cO\oplus \Lscr^{2}\oplus \Lscr^{3})$. Now
\[
\frac{1}{c(\Escr)}=\frac{1}{(1+2L)(1+3L)}=-2\cdot \frac{1}{1+2L} + 3\cdot \frac{1}{1+3L},
\]
and $C$ expanded as a series in $H$ is
\[
C=\beta_{0}+\beta_{1}H+\beta_{2}H^2+ \cdots,
\]
where the $\beta_{i}$'s are rational expressions in $L$, giving us our polynomial $f_C$ ($H^e=0$ for $e>dim(B)+2$). So in the notation of Theorem~\ref{mainthm}, $N=1$ so there is only one $f_{C_{j}}$, namely $f_{C_{1}}$, thus we can take
\[
f_{C_{1}}(H)=\beta_{2}+\beta_{3}H+\cdots +\beta_{2+dim(B)}H^{dim(B)}=\frac{C-(\beta_{0}+\beta_{1}H)}{H^2}.
\]
Therefore, by Theorem~\ref{mainthm}
\begin{eqnarray*}
\pi_{*}(C)&=&-2\cdot f_{c_{1}}(-2L)+3\cdot f_{c_{1}}(-3L)\\
          &=&-2\cdot\frac{C-(\beta_{0}+\beta_{1}(-2L))}{(-2L)^{2}}+3\cdot\frac{C-(\beta_{0}+\beta_{1}(-3L))}{(-3L)^{2}}\\
          &=&2\cdot \frac{6L}{1+6L},\\
\end{eqnarray*}
and since $g$ is a section of $\Lscr^{6}$,
\[
c(Z)=\frac{6L}{1+6L}c(B),
\]
thus the theorem is proved. Corollary 3.2 is obtained by plugging in $c_1(B)$ for $L$, then taking the degree of its zero dimensional component.\end{proof}
Note that in this case the formula from Theorem~\ref{mainthm} involved only two summands. We would also like to point out that in the proof above instead of using the polynomial $f_{C_1}$, we use the equivalent rational expression $\frac{C-(\beta_{0}+\beta_{1}H)}{H^2}$ (where the rational expression for $C$ is used), in hopes that the final result will \emph{look} like a Chern class as well. This is precisely what happens in all three cases of Theorem~\ref{CC}.\footnote[4]{For the $E_7$ case, we embed $Y$ as a complete intersection in an unweighted $\Pbb^{3}$ bundle over $B$, and then apply Theorem~\ref{mainthm}} Had we used the polynomial expression for $f_{C_1}$, $\pi_{*}(C)$'s identity as a Chern class of a divisor in the base would have been obscured by its polynomial form. Of course the result of Corollary 3.2 is but only \emph{one} of the Chern numbers of $Y$, the rest of which are just as easily obtained by application of Theorem~\ref{mainthm} (actually, due to the fact that $Y$ is Calabi-Yau, the only other non-zero Chern number is $\int c_2(Y)^2=\int24c_1(B)c_2(B)+120c_1(B)^3$).
\subsection{Plane curve fibrations of arbitrarily large genus}\label{HEF}
After proving Theorem~\ref{CC} using Theorem~\ref{mainthm}, we notice there was nothing special about the fact that the generic fiber of $Y$ was a degree $3$ plane curve of genus $1$. As such, we fully exploit the utility of Theorem~\ref{mainthm} and upgrade Theorem~\ref{CC} to a statement about plane curve fibrations of arbitrarily large genus, which we now consider.

Again, let $B$ be a non-singular compact complex algebraic variety endowed with a line bundle $\Lscr$, and let $\Escr=\cO\oplus \Lscr^{a}\oplus \Lscr^{b}$, for some $a,b\in \mathbb{Z}$. We consider a hypersurface $Y_{\vec{w}}$ in $\Pbb(\Escr)$ of class $[Y_{\vec{w}}]=dH+eL$ for some $d\in \Nbb, e\in \mathbb{Z}$, where $L=c_1(\Lscr)$, $H=c_1(\cO_{\Pbb(\Escr)}(1))$, and $\vec{w}=(a,b,d,e)$ is the vector of parameters on which our variety depends. Such a hypersurface determines a fibration $Y_{\vec{w}}\to B$ whose generic fiber is a plane curve of genus $g=\frac{(d-1)(d-2)}{2}$. We show that these fibrations are almost never Calabi-Yau, along with an extension of Theorem~\ref{CC}.

\begin{prop}\label{CYC}
Let $\varphi:Y_{\vec{w}}\to B$ be a plane curve fibration as defined above. Then $Y_{\vec{w}}$ cannot be Calabi-Yau for $d\neq3$. For $d=3$, $Y_{\vec{w}}$ is Calabi-Yau if and only if
\[
c_1(B)=(e-a-b)\cdot c_1(\Lscr).
\]
\end{prop}
\begin{proof}
A standard calculation yields
\[
K_{Y_{\vec{w}}}=\pi^{*}((e-a-b)\cdot L-c_1(B))+(3-d)\cdot c_1(\cO_{\Pbb(\Escr)}(1)).
\]
Thus $Y_{\vec{w}}$ is Calabi-Yau if and only if
\begin{equation*}
\tag{$\dagger$}
\pi^{*}((e-a-b)\cdot L-c_1(B))=(d-3)\cdot c_1(\cO_{\Pbb(\Escr)}(1))
\end{equation*}
But intersecting both sides of this equation with the class $F$ of a generic fiber yields
\[
0=(d-3)\cdot d[pt],
\]
which is obviously false for $d\neq3$. For $d=3$, the RHS of $(\dag)$ is $0$ so the proposition is proved.
\end{proof}

We now upgrade Theorem~\ref{CC}, as promised.
\begin{theorem}\label{CF}
Let $\varphi:Y_{\vec{w}}\to B$ be a plane curve fibration as defined above, and let $\Fscr=\Lscr^{e}\oplus \Lscr^{e-ad}\oplus \Lscr^{e-bd}$. Then
\[
\varphi_{*}c(Y_{\vec{w}})=X\cdot s(\Fscr)\cdot c(B),
\]
where $X=((3e^3-3bde^2-3ade^2+3abd^{2}e)\cdot L^{3}+(6e^2-4bde-4ade+2abd^2)\cdot L^{2}+(3de+3e-ad-ad^2-bd^2-bd)\cdot L+(3d-d^2))\in A^{*}B$, and $s(\Fscr)=c(\Fscr)^{-1}$ is the Segre class of $\Fscr$.

\end{theorem}
The proof is exactly the same as the proof given above for the $E_8$ case of Theorem~\ref{CC}, and is omitted. We stress the fact that the formula of Theorem~\ref{mainthm} is essential in the proof of this result, as it allows us to generate the class $X$ and the vector bundle $\Fscr$ as a function of the parameters $\vec{w}$. The $E_6$ and $E_8$ cases of Theorem~\ref{CC} are just the $(1,1,3,3)$ and $(2,3,3,6)$ cases of this result respectively. It is natural to ask wether or not there exist other $\vec{w}$s for which $\varphi_{*}c(Y_{\vec{w}})$ is supported on a Chern class of a divisor in the base other than $(1,1,3,3)$ and $(2,3,3,6)$. To give a partial answer, we multiply $X\cdot s(\Fscr)$ by $1+fL$ for $f\in \mathbb{Z}$, expand the result as a series in $L$, then set the first few coefficients of $L^i$ for $i\neq1$ equal to zero. Our solutions are then candidates for possible $\vec{w}$s. It turns out our solutions all lead to $\varphi_{*}c(Y_{\vec{w}})$s of the desired form. We list the results.
\\

\begin{center}
    \begin{tabular}{ | l | l | l | l | l | l | }
    \hline
$\vec{w}$ & $(\frac{f}{2},\frac{f}{4},3,f)$ & $(\frac{f}{2},\frac{f}{3},3,f)$ & $(\frac{f}{3},\frac{f}{3},3,f)$ & $(\frac{-f}{2},\frac{-f}{4},3,\frac{-f}{2})$ & $(\frac{-f}{2},\frac{-f}{6},3,\frac{-f}{2})$  \\ \hline
$\varphi_{*}c(Y_{\vec{w}})$ & $\frac{3fL}{1+fL}c(B)$ & $\frac{2fL}{1+fL}c(B)$ & $\frac{4fL}{1+fL}c(B)$ & $\frac{3fL}{1+fL}c(B)$ & $\frac{2fL}{1+fL}c(B)$ \\ \hline

    \end{tabular}
\end{center}

\subsection{A K3/elliptically fibered Calabi-Yau fourfold}\label{K3E}
We conclude this section with a pretty example inspired by F-theory/heterotic dualities. We construct a class of Calabi-Yau fourfolds $Y$ which are both elliptically fibered as well as $K3$ fibered, and compute their arithmetic genera using only adjunction and Theorem~\ref{mainthm}. Our starting point is a $K3$ surface, realized as an $E_8$ elliptic fibration as defined above with $B=\Pbb^1$:

\[
\xymatrix{
K3 \ar[r] \ar[d] & \Pbb(\cO\oplus \cO(4)\oplus \cO(6)) \ar[dl] \\
\Pbb^1
}
\]
Recall that an $E_8$ elliptic fibration is Calabi-Yau if and only if $c_1(\Lscr)=c_1(B)$, determining the choice $\Lscr=\cO_{\Pbb^1}(2)$. Next, we take a $\Pbb^1$-bundle $S=\Pbb(\cO\oplus \cO(n))\to \Pbb^2$, and construct an $E_8$ elliptic fibration $Y\to S$ over this threefold in such a way that the restriction of our fibration to each of the $\Pbb^1$ fibers of $S$ is our $K3$ constructed above. The resulting fibration $Y$ will be elliptically fibered over $S$ as well as $K3$ fibered over $\Pbb^2$. To accomplish this, we need a line bundle $\Mscr$ on $S$ such that we have the following commutative diagram for any fiber $\Pbb^1$ of $S$:
\[
\xymatrix{
K3 \ar[r] \ar[dr] & \Pbb(\cO\oplus \cO(4)\oplus \cO(6))=\Pbb(i^{*}\Escr) \ar[r] \ar[d] & \Pbb(\Escr) = \Pbb(\cO \oplus \Mscr^2 \oplus \Mscr^3) \ar[d]^{\pi_2}  & Y \ar[l] \ar[dl]\\
 & \Pbb^1 \ar[r]^{i} & S=\Pbb(\cO \oplus \cO(n)) \ar[d]^{\pi_1} \\
 & & \Pbb^2
}
\]
And for $Y$ to be Calabi-Yau, we also need $c_1(\Mscr)=c_1(S)=2K+\pi_{1}^{*}c_1(\cO(n))+\pi_{1}^{*}c_1(\Pbb^2)$, where $K=c_1(\cO_{\Pbb(\Fscr)}(1))$ and $\Fscr=\cO\oplus \cO(n)$. These two conditions on $\Mscr$ are achieved by taking $\Mscr=\pi_{1}^{*}\cO_{\Pbb^2}(3)\otimes \cO_{S}(2)\otimes \pi_{1}^{*}\cO_{\Pbb^{2}}(n)$ (the former condition is satisfied due to the fact that pullback bundles don't affect the geometry of $\Pbb(\Escr)$ as we move along the individual $\Pbb^1$ fibers in the base $S$, so that $\Pbb(i^{*}\Escr)=\Pbb(\cO\oplus \cO(4)\oplus \cO(6))$). Repeating this construction with an $E_6$ or $E_7$ $K3$ surface as our starting point yields two more classes of examples. To compute their arithmetic genera, we first compute their Chern numbers via Theorem~\ref{mainthm}. For example in the $E_6$ case, we have that

\[
c(Y)=\frac{(1+H)(1+H+M)^{2}(3H+3M)\pi_{2}^{*}c(S)}{1+3H+3M}\in A^{*}\Pbb(\Escr),
\]
where $M=c_1(\Mscr)$, and $H=c_1(\cO_{\Pbb(\Escr)}(1))$. Expanding this expression as a series in $H$ immediately yields

\begin{eqnarray*}
c_2^{2}=\alpha_{0}+\alpha_{1}H+\cdots +\alpha_{5}H^5\in A^{5}\Pbb(\Escr), & \\
c_4=\beta_{0}+\beta_{1}H+\cdots +\beta_{5}H^5\in A^{5}\Pbb(\Escr), & \\
\end{eqnarray*}
where the $\alpha_{i}$s and $\beta_{j}$s are polynomial expressions in $M$ and the Chern classes of $S$. Using the fact that in the $E_6$ case $\frac{1}{c(\Escr)}=\frac{1}{(1+M)^2}$, we first pushforward these zero-cycles to $S$ using Theorem~\ref{mainthm}:

\begin{eqnarray*}
{\pi_{2}}_{*}(c_2^{2})&=& \frac{d}{dH} \left(\frac{c_2^{2}-(\alpha_{0}+\alpha_{1}H)}{H}\right)|_{H=-M}  \\
                      &=&24c_1(S)c_2(S)+24c_1(S)^3  \\
                      &=&a_{0}+a_{1}K+a_{2}K^2+a_{3}K^3,  \\
{\pi_{2}}_{*}(c_4)&=& \frac{d}{dH} \left(\frac{c_4-(\beta_{0}+\beta_{1}H)}{H}\right)|_{H=-M}  \\
                  &=&12c_1(S)c_2(S)+72c_1(S)^3  \\
                  &=&b_{0}+b_{1}K+b_{2}K^2+b_{3}K^3, \\
\end{eqnarray*}
where the $a_{i}s$ and $b_{j}s$ are polynomials in $h=c_1(\cO_{\Pbb^2}(1))$ of degrees $3-i$ and $3-j$ respectively. The final two equalities for each expression come from the fact that
\[
M=c_1(S)=(n+3)h+2K, \hspace{5mm}c_2(S)=3(n+1)h^2+(n+6)h\cdot K+K^2.
\]
Then using the fact that $\frac{1}{c(\Fscr)}=\frac{1}{(1+nh)}$, we apply Theorem~\ref{mainthm} once again to get that

\begin{eqnarray*}
{\pi_{1}}_{*}({\pi_{2}}_{*}(c_2^{2}))&=&\left( \frac{{\pi_{2}}_{*}(c_2^{2})-a_0}{K}\right)|_{K=-nh} \\
                                     &=&(1872+48n^2)h^2, \\
{\pi_{1}}_{*}({\pi_{2}}_{*}(c_4))&=&\left( \frac{{\pi_{2}}_{*}(c_4)-b_0}{K}\right)|_{K=-nh} \\
                                 &=&(4176+144n^2)h^2. \\
\end{eqnarray*}
And since $\int \pi_{*}(C)=\int C$ for any proper pushforward $\pi_{*}$, we get that
\[
\int_{\Pbb(\Escr)}c_2(Y)^2=1872+48n^2, \hspace{5mm}\int_{\Pbb(\Escr)}c_4(Y)=4176+144n^2.
\]
Repeating this process for the $E_7$ and $E_8$ cases yields the following non-zero Chern numbers for all fibrations considered in this example:
\\
\begin{center}
    \begin{tabular}{ | l | l | l | l | }
    \hline
Chern numbers & $E_6$ & $E_7$ & $E_8$ \\ \hline
    $\int c_2(Y)^2$ & $1872+48n^2$ & $3168+96n^2$ & $7056+240n^2$ \\ \hline
    $\int c_4(Y)=\chi(Y)$ & $4176+144n^2$ & $8064 + 288n^2$ & $19728+720n^2$ \\ \hline

    \end{tabular}
\end{center}

With Chern numbers in hand and using the fact that $H^{p}(Y,\Omega^{q})=H^{q,p}(Y)$, we immediately compute the arithmetic genera $\chi_{q}:=\displaystyle\sum_{p=0}^{4}(-1)^{p}h^{q,p}(Y)=\int_{\Pbb(\Escr)}ch(\Omega^{q})Td(Y)$ (the second equality follows by Hirzebruch-Riemann-Roch):
\\

\begin{center}
    \begin{tabular}{ | l | l | l | l | }
    \hline
         arithmetic genera & $E_6$ & $E_7$ & $E_8$ \\ \hline
    $\chi_0=\frac{1}{720}\int3c_{2}^{2}-c_4$ & $2$ & $2$ & $2$ \\ \hline
    $\chi_1=\frac{1}{180}\int3c_{2}^{2}-31c_4$ & $-688-24n^2$ & $-1336 - 48n^2$ & $-3280-120n^2$\\ \hline
    $\chi_2=\frac{1}{120}\int3c_{2}^{2}+79c_4$ & $2796+96n^2$ & $5388+192n^2$ & $13164+480n^2$\\ \hline
    \end{tabular}
\end{center}
The fact that $\chi_0=2$ in all the cases above reflects the well known fact that for Calabi-Yau fourfolds
with full $SU(4)$ holonomy we have that $h^{0,0}=1$, $h^{1,0}=h^{2,0}=h^{3,0}=0$, and $h^{4,0}=1$. We note that no intersection numbers were computed in this example, as Theorem~\ref{mainthm} reduced the computation of Chern numbers to reading off a coefficient of $h^2$ in $A^{*}\Pbb^2$.

\section{On the Structure of $A_{*}\Pbb(\Escr)$}\label{structure}
We now take time to elaborate more on the structure of the Chow \emph{group} $A_{*}\Pbb(\Escr)$ for a general projective bundle over an algebraic scheme, and discuss how the hypotheses of Theorem~\ref{mainthm} can be eased to apply to this more general setting.

Let $B$ be an algebraic scheme over a field, $\Escr\to B$ be a vector bundle of rank $(n+1)$, and $\Pbb(\Escr)\overset \pi \to B$ its associated projectivization. If $B$ is singular, $\Pbb(\Escr)$ might acquire singularities as well, in which case intersection products of algebraic cycles are not well defined, denying us a commutative ring structure on the Chow group $A_{*}\Pbb(\Escr)$. However, the structure theorem for the Chow group of a projective bundle tells us that we can still \emph{think} of classes in $A_{*}\Pbb(\Escr)$ as polynomials in powers of $c_1(\cO_{\Pbb(\Escr)}(1))$ with coefficients in $A_{*}B$ none the less. The precise statement is the following:
\begin{theorem}(\cite{fulton})With notation and assumptions
as above, each element $\beta\in A_{k}\Pbb(\Escr)$ is uniquely expressible in the form
\[
\beta=\displaystyle\sum_{i=1}^{n}c_{1}(\cO_{\Pbb(\Escr)}(1))^{i}\cap \pi^{*}\alpha_{i},
\]
for $\alpha_{i}\in A_{k-n+i}(B)$.
\end{theorem}

In light of this fact, we know the map $p:A_{*}B[\zeta]\rightarrow A_{*}\Pbb(\Escr)$ given by
\[
\beta_0+\beta_{1}\zeta+\cdots+\beta_{r}\zeta^{r}\longmapsto \pi^{*}\beta_0+c_1(\cO_{\Pbb(\Escr)}(1))\cap \pi^{*}\beta_{1}+\cdots + c_1(\cO_{\Pbb(\Escr)}(1))^{r}\cap \pi^{*}\beta_{r}
\]
is a surjective morphism of groups. If $B$ (and so $\Pbb(\Escr)$) is non-singular, we can lift this statement to the level of Chow cohomology and deduce that any class $C\in A^{*}\Pbb(\Escr)$ can be written as
\[
C=\pi^{*}\beta_0+\pi^{*}\beta_{1}\cdot \zeta+\cdots+\pi^{*}\beta_{r}\cdot \zeta^{r},
\]
where $\zeta$=$c_1(\cO_{\Pbb(\Escr)}(1))\cap [\Pbb(\Escr)]$. Thus
\[
\pi_{*}(C)=\beta_{1}\cdot \pi_{*}(\zeta)+\cdots +\beta_{r}\cdot \pi_{*}(\zeta^r)
\]
by the projection formula, and so $\pi_{*}(C)$ is completely determined by the $\pi_{*}(\zeta^i)$s, which is absolutely essential for the proof of Theorem~\ref{mainthm} (as we will see in section \S\ref{proof}). Now if $B$ is singular, the projection formula doesn't apply to a general class of the form
\[
C=\pi^{*}\beta_0+c_1(\cO_{\Pbb(\Escr)}(1))\cap \pi^{*}\beta_{1}+\cdots + c_1(\cO_{\Pbb(\Escr)}(1))^{r}\cap \pi^{*}\beta_{r}
\]
as given by the structure theorem. However, if the $\pi^{*}\beta_i=q_{i}\cap [\Pbb(\Escr)]$ in the expression given above, where the $q_{i}=Q_{i}(c_{j_{1}}(\pi^{*}E_{1}),\cdots ,c_{j_{m}}(\pi^{*}E_{m}))$ are polynomials in Chern classes of vector bundles on $B$, then we can write $c_1(\cO_{\Pbb(\Escr)}(1))^{i}\cap \pi^{*}\beta_{i}$ as
\[
q_{i}\cdot c_1(\cO_{\Pbb(\Escr)}(1))^{i}\cap [\Pbb(\Escr)]=q_i\cap \zeta^i.
\]
Now by the projection formula, $\pi_{*}(q_i\cap \zeta^i)=q_{i}\cap \pi_{*}(\zeta^i)$, and so $\pi_{*}(C)$ again only depends on the $\pi_{*}(\zeta^i)$s as in the non-singular case, thus we can apply the formula of Theorem~\ref{mainthm} to compute $\pi_{*}(C)$. In summary, Theorem~\ref{mainthm} applies to classes in the Chow group of a projective bundle of the form stated above over \emph{any} algebraic scheme $B$, long as $C$ is written in terms of polynomials of Chern classes of vector bundles on $B$ operating on powers of the hyperplane section in $\Pbb(\Escr)$.

\section{The Proof}\label{proof}
We recall the assumptions made in \S\ref{intro}. Let $B$ be a non-singular compact complex algebraic variety, $\Lscr$ a line bundle on $B$, $\Escr\to B$ be a vector bundle of the form $\Escr=\Lscr^{a_1}\oplus \cdots \oplus \Lscr^{a_{n+1}}$, and let $\Pbb(\Escr)\overset \pi\to B$ be its associated projectivization. We prove Theorem~\ref{mainthm} in the case that only one non-zero power of $\Lscr$ appears in $\Escr$, i.e., $c(\Escr)=(1+aL)^k$ for some $a\in \mathbb{Z}$, where $L=c_1(\Lscr)$. Once this case is established the general case will immediately follow.
\begin{proof}
Let $C\in A^{*}\Pbb(\Escr)$, and let $f_{C}= \beta_{0}+\cdots+\beta_{s}\zeta^{s}$ be any polynomial which maps to $C$ under the natural projection $p:A^{*}B[\zeta]\to A^{*}\Pbb(\Escr)$, i.e., $f_C$ is a polynomial representation of $C$ in terms of powers of $\zeta=\cO_{\Pbb(\Escr)}(1)$ . Since $\pi_{*}(\pi^{*}\beta_{i}\cdot \zeta^{i})=\beta_i\cdot \pi_{*}(\zeta^{i})$, $\pi_{*}(C)$ is completely determined by the $\pi_{*}(\zeta^i)$s. To evaluate $\pi_{*}(\zeta^i)$, we use the \emph{Segre class} $s(\Escr)$ of $\Escr$, which is the multiplicative inverse of $c(\Escr)$ in $A^{*}\Pbb(\Escr)$. By definition,
\[
s(\Escr)=\pi_{*}(1+\zeta+\zeta^{2}+\cdots ),
\]
and since $s(\Escr)=\frac{1}{c(\Escr)}$,
\begin{eqnarray*}
\pi_{*}(1+\zeta+\zeta^{2}+\cdots )&=&\frac{1}{(1+aL)^k}\\
                                  &=&\frac{1}{(k-1)!}\cdot \frac{d^{k-1}}{dx^{k-1}}\left(\frac{1}{1-x}\right)|_{x=-aL}\\
                                  &=&(1+\alpha_{1}x+\alpha_{2}x^2+\cdots)|_{x=-aL}.\\
\end{eqnarray*}
Matching terms of like dimension, we see that
\[
1\mapsto0,\hspace{2mm}\zeta\mapsto0,\cdots, \zeta^{n-1}\mapsto0,\hspace{2mm}\zeta^{n}\mapsto1,\hspace{2mm}\zeta^{n+1}\mapsto\alpha_{1}(-aL),\hspace{2mm}\zeta^{n+2}\mapsto\alpha_{2}(-aL)^2,
\cdots
\]

So to obtain $\pi_{*}(C)$, we must take $f_C=\beta_{0}+\cdots+\beta_{s}\zeta^{s}$ and substitute the $\zeta^{i}$s according to the rules above. To accomplish this, we introduce the polynomial
\begin{eqnarray*}
f_{C_{k}}&=&\frac{1}{\zeta^{n-k+1}}\cdot (f_C-(\beta_{0}+\cdots +\beta_{n-1}\zeta^{n-1}))\\
         &=&\beta_{n}\zeta^{k-1}+\cdots+\beta_{s}\zeta^{s-(n-k+1)}.\\
\end{eqnarray*}
Thus,
\[
\frac{1}{(k-1)!}\cdot \frac{d^{k-1}}{d\zeta^{k-1}}(f_{C_{k}})=\beta_{n}+\beta_{n+1}\cdot (\alpha_{1}\zeta)+\cdots+\beta_{s}\cdot (\alpha_{l}\zeta^{l}),
\]
where $l=s-n$. The crux of the proof is that evaluating this expression at $\zeta=-aL$ performs all the desired substitutions in one fell swoop, i.e.,
\[
\pi_{*}(C)=\frac{1}{(k-1)!}\cdot \frac{d^{k-1}}{d\zeta^{k-1}}(f_{C_{k}})|_{\zeta=-aL},
\]
which is precisely the statement of Theorem~\ref{mainthm} in our case.
For the general case where $c(\Escr)=(1+d_{1}L)^{k_{1}}\cdots(1+d_{m}L)^{k_{m}}$, just take the partial fraction decomposition of $\frac{1}{c(\Escr)}$ and reduce the computation to a linear combination of cases as above.
\end{proof}

\section{Further Directions}
Exercising one's (perhaps) natural penchant for generalization, there are obvious directions in which the scope of Theorem~\ref{mainthm} could be furthered. Products of projective bundles each of which are of the form assumed throughout this note would be a natural starting place for one. It would also be nice to have such a formula for \emph{weighted} projective bundles. In \cite{amrani}, al Amrani generalizes the relation we have in the Chow cohomology of $\Pbb(\Escr)$ coming from the unique monic generator of the kernel of the map $p:A^{*}B[\zeta]\to A^{*}\Pbb(\Escr)$, namely
\begin{equation*}
\tag{$\dagger \dagger$}
\zeta^{n}+c_1(\Escr)\zeta^{n-1}+\cdots +c_{n}(\Escr)=0
\end{equation*}
where $n=rk(\Escr)$ and $\zeta=c_1(\cO_{\Pbb(\Escr)}(1))$, to an analogous relation in the topological cohomology ring of a weighted projective bundle $\Pbb_{\vec{w}}(\Escr)$, which is obtained from $(\dag \dag)$ by replacing each $c_{i}(\Escr)$ by his "twisted" Chern classes relative to the weights $\vec{w}$ of the projectivization (he defines an analogous tautological bundle as well). In the proof of Theorem~\ref{mainthm}, we used the Segre class $s(\Escr)$ to determine the pushforwards of the $\pi_{*}(\zeta^{i})$s. Alternatively, we could have used $(\dag \dag)$ to determine the $\pi_{*}(\zeta^{i})$s, so al Amrani's generalization of $(\dag \dag)$ could perhaps be used to determine pushforwards of powers of hyperplane sections in $A_{*}\Pbb_{\vec{w}}(\Escr)$ as well. Going even further, we speculate that a generalization at the level of \emph{toric} bundles lurks in the shadows.

%\bibliographystyle{abbrv}
%\bibliography{paperbib}

\begin{thebibliography}{1}

\bibitem{Vafa}
C.Vafa.
\newblock Evidence for F-theory
\newblock {\em Nuclear Phys. B}, 469(3):403--415, 1996.

\bibitem{Denef les houches}
F. Denef.
\newblock Les Houches Lectures on Constructing String Vacua
\newblock arXiv:0803.1194.

\bibitem{amrani}
A.~Al~Amrani.
\newblock Cohomological study of weighted projective spaces.
\newblock In {\em Algebraic geometry ({A}nkara, 1995)}, volume 193 of {\em
  Lecture Notes in Pure and Appl. Math.}, pages 1--52. Dekker, New York, 1997.

\bibitem{aluffi}
P.~Aluffi and M.~Esole.
\newblock New orientifold weak coupling limits in {F}-theory.
\newblock {\em J. High Energy Phys.}, (2):020, i, 52, 2010.

\bibitem{fulton}
W.~Fulton.
\newblock {\em Intersection theory}, volume~2 of {\em Ergebnisse der Mathematik
  und ihrer Grenzgebiete (3) [Results in Mathematics and Related Areas (3)]}.
\newblock Springer-Verlag, Berlin, 1984.

\bibitem{CY4}
A.~Klemm, B.~Lian, S.-S. Roan, and S.-T. Yau.
\newblock Calabi-{Y}au four-folds for {M}- and {F}-theory compactifications.
\newblock {\em Nuclear Phys. B}, 518(3):515--574, 1998.

\bibitem{SVW}
S.~Sethi, C.~Vafa, and E.~Witten.
\newblock Constraints on low-dimensional string compactifications.
\newblock {\em Nuclear Phys. B}, 480(1-2):213--224, 1996.

\bibitem{timo}
T.~Weigand.
\newblock Lectures on {F}-theory compactifications and model building.
\newblock arXiv:1009.3497.

\end{thebibliography}

\end{document}